\documentclass[11pt]{article}

 \usepackage{amsmath,amsthm,amssymb}
 \usepackage{makeidx,epsfig,lscape}
 \usepackage{color,colortbl}
 \usepackage{fancyhdr}
 \usepackage{times}
\usepackage{indentfirst}
 \usepackage{xcolor,pict2e}
 \usepackage[latin1]{inputenc}
 
\newcommand{\R}{\mathbb R}
\newcommand{\N}{\mathbb N}

\newcommand{\Qv}{\mathbb Q}
\newcommand{\Pv}{\mathbb P}
\newcommand{\E}{\mathbb E}

\newcommand{\gt}{\gamma_{t}}
\newcommand{\gs}{\gamma_{s}}
\newcommand{\g}{\gamma_{\cdot}}

 \renewcommand{\headrulewidth}{0pt}
 \renewcommand{\footrulewidth}{0.5pt}

 \definecolor{myaqua}{rgb}{0.0,0.5,0.55}
 \definecolor{lightaqua}{rgb}{0.75,0.95,0.95}

 \usepackage[colorlinks = true,
            linkcolor = blue,
            urlcolor  = blue,
            citecolor = red,
            pdfpagemode=UseOutlines,
            bookmarksnumbered=true,
            pdfpagelabels,
            breaklinks]{hyperref}

\usepackage{caption}
\usepackage{floatrow}

\newtheorem{theorem}{Theorem}
\newtheorem{prop}{Proposition}
\newtheorem{lem}{Lemma}
\newtheorem{coro}{Corollary}
\newtheorem{defn}{Definition}[section]
\newtheorem{rem}{Remark}[section]
\def\lin#1#2{\textcolor[rgb]{0.6,0.6,0.6}{\vspace*{#1mm} \hrule
   height 3 pt \vspace*{#2mm}}}
\def\bt{\begin{tabular}}
\def\et{\end{tabular}}
\def\and{\mbox{ and }}
\def\E{\mbox{\bf E}}
\def\P{\mbox{\bf P}}

\def\1{{\bf 1}}

 \def\boxx#1#2#3#4#5{
 {\linethickness{#4pt}\put(#1,#5){\color{myaqua}{\line(1,0){#3}}}}
 \multiput(#1,#2)(0,#4){2}{\line(1,0){#3}}
 \multiput(#1,#2)(#3,0){2}{\line(0,1){#4}}
  }

\begin{document}

 $\mbox{ }$

 \vskip 12mm

{ 
{\noindent{\Large\bf\color{myaqua}
 Resolution of the skew Brownian motion  equations with stochastic calculus for signed measures}} 
\\[6mm]
{\bf Fulgence EYI OBIANG }}
\\[2mm]
{ 
URMI Laboratory, Département de Mathématiques et Informatique, Faculté des Sciences, Université des Sciences et Techniques de Masuku,  Franceville, Gabon 
\\
Email: \href{mailto:feyiobiang@yahoo.fr}{\color{blue}{\underline{\smash{feyiobiang@yahoo.fr}}}}\\[1mm]
\lin{5}{7}

 {  
 {\noindent{\large\bf\color{myaqua} Abstract}{\bf \\[3mm]
 \textup{
Contributions of the present paper consist of two parts. In the first one, we contribute to the theory of stochastic calculus for signed measures. For instance, we provide some results permitting to characterize martingales and Brownian motion both defined under a signed measure. We also prove that the uniformly integrable martingales (defined with respect to a signed measure) can be expressed as relative martingales and we provide some new results to the study of the class $\Sigma(H)$ which appeared for the first time in \cite{f} and studied in \cite{f,e,o}. The second part is devoted to the construction of solutions for the \textbf{homogeneous skew Brownian motion equation} and for the \textbf{inhomogeneous skew Brownian motion equation}. To do this, our ingredients are the techniques and results developed in the first part that we apply on some stochastic processes borrowed from the theory of stochastic calculus for signed measures. Our methods are inspired by those used by Bouhadou and Ouknine in \cite{siam}. Moreover, their solution of the inhomogeneous skew Brownian motion equation is a particular case of those we propose in this paper.
 }}} 
 \\[4mm]
 {\noindent{\large\bf\color{myaqua} Keywords}{\bf \\[3mm]
 Stochastic calculus for signed measures; Skew Brownian motion; class $\Sigma(H)$; relative martingales; honest time; zeros of continuous martingales 
}}}\\[4mm]{\noindent{\large\bf\color{myaqua} MSC: }{\color{blue} 60G07; 60G20; 60G46; 60G48;60H10; 60J60}}
\lin{3}{1}

\renewcommand{\headrulewidth}{0.5pt}
\renewcommand{\footrulewidth}{0pt}

 \pagestyle{fancy}
 \fancyfoot{}
 \fancyhead{} 
 \fancyhf{}
 \fancyhead[RO]{\leavevmode \put(-100,0){\color{myaqua}Fulgence EYI OBIANG} \boxx{15}{-10}{10}{50}{15} }
 \fancyfoot[C]{\leavevmode
 \put(-2.5,-3){\color{myaqua}\thepage}}

 \renewcommand{\headrule}{\hbox to\headwidth{\color{myaqua}\leaders\hrule height \headrulewidth\hfill}}
\section*{Introduction}

  The study of stochastic calculus when the measure space is governed, not by a probability measure, but by a general measure that can take positive and negative values (signed measure) is called \textbf{stochastic calculus for signed measures}. It is a relatively recent theory. The first publications on it go back to 1984 (Ruize de Chavez, \cite{chav}) and 2003 (Beghdadi-Sakrani, \cite{sak}). It stems from the suggestion of P.A Meyer who proposed to generalize in the field of signed measures, Paul Lévy's theorem which characterizes the measure of Wiener on $\Omega=\mathcal{C}_{0}(\R_{+},\R)$ as the unique probability under which $X$ and $(X^{2}_{t}-t:t\geq0)$ are martingales, where $X$ is the canonical process on $\Omega$. The bases of this theory have been implemented in the two above mentioned papers. For example, martingale theory for signed measures is developed and some classical results of the usual stochastic calculus were generalized. One can quote for instance Theorem of Itô and Theorem of Girsanov or Theorem of Paul Lévy. Recently, the theory of Stochastic Calculus  for Signed Measures has registered new contributions in a series of three publications of Eyi Obiang (the first two of which are in collaboration with Ouknine and Moutsinga \cite{f,e}, the third with Ouknine, Moutsinga and Trutnau \cite{o}). Globally, these three papers are devoted to the development of a framework and techniques permitting to study  stochastic processes of two new classes of stochastic processes namely, class $\Sigma_{s}(H)$ and class $\Sigma(H)$ appeared for the first time in \cite{f} .

The present paper aims two main objectives. The first is to bring substantial contributions to develop the theory of stochastic calculus for signed measures. For this purpose, we contribute to the theory of martingales  by providing  new  results permitting to characterize stochastic processes that Ruiz de Chavez designates as martingales and Brownian motions  under a signed measure. Some corollaries of these results prove that we can construct martingales and Brownian motions (In the sense of the signed measures and in the sense of probabilities). We show that some uniformly integrable martingales with respect to a signed measure are in fact relative martingales. That is, we can write them as
$$M_{t}=E\left[M_{\infty}1_{\{g<t\}}|\mathcal{F}_{t}\right]$$
where $g$ is an honest time enjoying a capital role in stochastic calculus for signed measures. We contribute also to the study of the class $\Sigma(H)$ by giving new characterization results and new interesting properties.

The second main aim of this manuscript is devoted to the construction of solutions for the homogeneous skew Brownian Motion and for the inhomogeneous skew Brownian motion. That is, we construct solutions for  the two following equations
\begin{equation}\label{hsm}
	X_{t}=x+B_{t}+(2\alpha-1)L_{t}^{0}(X)
\end{equation}
where $B$ is a standard Brownian motion, $\alpha\in(0,1)$ is a skewness parameter, $x\in\R$ and $L_{t}^{0}(X)$ stands for the symmetric local time at $0$. And
\begin{equation}\label{ism}
	X^{\alpha}_{t}=x+B_{t}+\int_{0}^{t}{(2\alpha(s)-1)dL_{s}^{0}(X^{\alpha})}
\end{equation}
where $B$ is a standard Brownian motion, $x\in\R$, $\alpha:\R_{+}\rightarrow\R$ is a Borel function and $L_{t}^{0}(X^{\alpha})$ stands for the symmetric local time at $0$ of the unknown process $X^{\alpha}$.

A strong solution of \eqref{hsm} is called skew Brownian motion. It appeared in the seminal work \cite{11} of Itô and McKean as a natural generalization of the Brownian motion. It is a process that behaves like a Brownian motion except that the sign of each excursion is chosen using an independent Bernoulli random variable of parameter $\alpha$. This equation has been studied extensively and has many extensions in the literature. For instance, we can quote Harrison and Shepp \cite{10}, LeGall \cite{12}, Ouknine \cite{16} and Walsh \cite{21}. The inhomogeneous skew Brownian motion (Equation \eqref{ism}) is one of these extensions. It was introduced by Weinryb \cite{23} and recently studied by Etoré and Martinez in \cite{9} and by Bouhadou and Ouknine in \cite{siam}.

The paper is organized as follows:
\begin{itemize}
	\item In section \ref{s1}, we first present notations and definitions useful in this paper. We also recall some results on enlargement of filtration and balayage formula we use throughout this work.
	\item In Section, \ref{s2}, we contribute to the theory of martingales for signed measures and to the study of Brownian motion defined with respect to a signed measure.
	\item In Section \ref{s4}, we give new properties and new characterization results for stochastic processes of class $\Sigma(H)$.
	\item The section \ref{s5} is dedicated to the construction of solutions of skew Brownian motion equations from stochastic processes borrowed from the theory of stochastic calculus for signed measures.
\end{itemize}

\section{Preliminaries}\label{s1}

The present section is reserved to the presentation of  some preliminaries on which the theory set out in this paper is based.

\subsection{Notations}

We start by giving some notations which will be used throughout this paper. Consider a measure space $(\Omega, \mathcal{F}_{\infty}, \Qv)$, where $\Qv$ is a bounded signed measure. Let $\P$ be a probability measure on $\mathcal{F}_{\infty}$ such that $\Qv\ll\P$ and $(\mathcal{F}_{t})_{t\geq0}$ be a right continuous filtration completed with respect to $\P$ such that $\mathcal{F}_{\infty}=\vee_{t}{\mathcal{F}_{t}}$. We shall use the following notations:
\begin{itemize}
	\item $D_{t}=\frac{d\Qv}{d\P}|_{\mathcal{F}_{t}}$ where $(\mathcal{F})_{t\geq0}$ is a right continuous filtration completed with respect to $\P$ such that $\mathcal{F}_{\infty}=\vee_{t}\mathcal{F}_{t}$. Note that $D$ is a uniformly integrable $\P$- martingale (see Beghdadi-Sakrani \cite{sak}).
	\item $H=\{t: D_{t}=0\}$;
	\item $g=\sup{H}$; $\overline{g}=0\vee g$ and $\gamma_{t}=\sup\{s\leq t: D_{s}=0\}$;
	\item Remark that $g$ and $\overline{g}$ are not stopping times but honest times. The smallest right continuous filtration containing $(\mathcal{F}_{t})$ for which $\overline{g}$ is a stopping time will be denoted by $(\mathcal{G}_{t})$.
 Then, the filtration $(\mathcal{G}_{\overline{g}+t})$ is well defined.
	\item $\P^{'}=\frac{|D_{\infty}|}{\E(|D_{\infty}|)}\P$.
\end{itemize}

In this work, we shall assume that $D$ is a continuous process which satisfies $\P(D_{\infty}=0)=0$ (i.e $\overline{g}<\infty$ almost surely).

We close this subsection by defining two stochastic processes which are very important for the present work. For any continuous semi-martingale $Y$, the set $\mathcal{W}=\{t\geq0; Y_{t}=0\}$ cannot be ordered. However, the set $\R_{+}\setminus\mathcal{W}$ can be decomposed as a countable union $\cup_{n\N}{J_{n}}$ of intervals $J_{n}$. Each interval $J_{n}$ corresponds to some excursion of $Y$. That is if $J_{n}=]g_{n},d_{n}[$, $Y_{t}\neq0$ for all $t\in]g_{n},d_{n}[$ and $Y_{g_{n}}=Y_{d_{n}}=0$. At each $J_{n}$ we associate a Bernoulli random variable $\zeta_{n}$ which is independent from any other random variables and such that 
$$P(\zeta_{n}=1)=\alpha\text{ and }P(\zeta_{n}=-1)=1-\alpha.$$
Now, let us define the process $Z^{\alpha}$ we use throughout this paper.
\begin{equation}\label{zalpha}
	Z^{\alpha}_{t}=\sum_{n=0}^{+\infty}{\zeta_{n}1_{]g_{n},d_{n}[}(t)}.
\end{equation}

If we consider that $\alpha$ is a piecewise constant function associated with a partition $(0=t_{0}<t_{1}<\cdots<t_{n-1}<t_{m})$. That is, $\alpha$ is of the form:
$$\alpha(t)=\sum_{i=0}^{m}{\alpha_{i}1_{[t_{i},t_{i+1})}(t)},$$
where $\alpha_{i}\in[0,1]$ for all $i=0,1,\cdots,n$. In this case, we shall consider the following process
\begin{equation}\label{Zalpha}
	\mathcal{Z}^{\alpha}_{t}=\sum_{n=0}^{+\infty}{\sum_{i=0}^{m}{\zeta^{i}_{n}1_{]g_{n},d_{n}[\cap[t-{i},t_{i+1})}(t)}},
\end{equation}
where $(\zeta^{i}_{n})_{n\geq0}$, $i=1,2,\cdots,m$ be $m$ independent sequences of independent variables such that
$$\P(\zeta^{i}_{n}=1)=\alpha_{i}\text{ and }\P(\zeta^{i}_{n}=-1)=1-\alpha_{i}.$$
\subsection{Enlargement of filtrations}

Now, we shall recall some results of the theory of enlargement of filtrations which are mainly  useful in Section \ref{s22} of the current work.

\begin{defn}[\textbf{Azéma and Yor\cite{1}}]
Let $\mathcal{H}$ be a random  optional closed set. We call $\mathcal{R}(\mathcal{H})$ the class of processes $(X_{t};t\geq0)$ vanishing on $\mathcal{H}$ and admitting a decomposition of the form
$$X_{t}=M_{t}+V_{t},$$
where $(M_{t};t\geq0)$ is a right continuous uniformly integrable martingale, $(V_{t};t\geq0)$ is a continuous and adapted variation integrable process such that $dV_{t}$ is carried by $\mathcal{H}$.
\end{defn}

\begin{prop}[\textbf{Azéma and Yor\cite{1}}]\label{rho}

Let $\Gamma$ be an honest time with respect to the filtration $(\mathcal{F}_{t})_{t\geq0}$ and a closed optional set $\mathcal{H}$ such that $\Gamma=\sup{(\mathcal{H})}$. We represent by $(\mathcal{G}_{t})_{t\geq0}$, the progressive enlargement of the filtration $(\mathcal{F}_{t})_{t\geq0}$ with respect to $\Gamma$. Let $(V_{t})_{t\geq0}$ be a $(\mathcal{G}_{\Gamma+t})_{t\geq0}-$   optional process. There exists a unique $(\mathcal{F}_{t})_{t\geq0}-$   optional process $(U_{t})_{t\geq0}$  which vanishes on $\mathcal{H}$ such that $\forall t\geq0$, $U_{\Gamma+t}=V_{t}$ and $U_{0}=V_{0}$ on $\{\Gamma=0\}$. That defines a function
$\rho:V\longmapsto U$. $\rho$ is linear, non-negative and preserves products.
\end{prop}       

\begin{theorem}\label{qot}[\textbf{Quotient theorem: Theorem 3.2 of Azéma and Yor \cite{1}}]
Let $Y$ be a non-negative  right continuous process such that $Y\in\mathcal{R}(\mathcal{H})$ with $\mathcal{H}=\{t\geq 0:Y_{t}=0\}$. Denote $\Gamma=\sup{(\mathcal{H})}$.
\begin{enumerate}
	\item If $(X_{t};t\geq0)$ is a stochastic process of the class $\mathcal{R}(\mathcal{H})$, then, the process $(\chi_{t};t>0)$ defined by 
$$\chi_{t}=\frac{X_{\Gamma+t}}{Y_{\Gamma+t}}$$
is a $\left(Q, (\mathcal{G}_{\Gamma+t}){t>0}\right)$ uniformly integrable martingale.
\item Reciprocally, let $(\chi_{t};t>0)$ a $\left(Q, (\mathcal{G}_{\Gamma+t}){t>0}\right)$ uniformly integrable martingale; the stochastic process $X=(Y_{t}\rho(\chi_{\cdot})_{t};t\geq0)$ is the unique process of $\mathcal{R}(H)$ such that
$$\chi_{t}=\frac{X_{\Gamma+t}}{Y_{\Gamma+t}}$$
for all $t>0$.
\end{enumerate}
Where, $Q=\dfrac{Y_{\infty}}{E(Y_{\infty})}\P$.
\end{theorem}

\begin{lem}[\textbf{Lemma 5.7 of Jeulin \cite{jeulin80}}]\label{j80}
Let $\Gamma$ be an honest variable with respect to $(\mathcal{F}_{t})_{t\geq0}$. Let $(\mathcal{G}_{t})_{t\geq0}$ be the progressive enlargement of the filtration $(\mathcal{F}_{t})_{t\geq0}$ with respect to $\Gamma$. If $\tau$ is a stopping time with respect to $(\mathcal{G}_{t})_{t\geq0}$ such that $\Gamma<\tau$ on $\{\Gamma<\infty\}$, hence
$$\mathcal{G}_{\tau}=\mathcal{F}_{\tau}.$$
\end{lem}

\subsection{Balayage formula}
The balayage formula is a key tool very used in this work. In what follows, we recall results we used. We begin by giving the balayage formula in the predictable case.

\begin{prop}
Let $Y$ be a continuous semi-martingale and $\gamma^{'}_{t}=\sup\{s\leq t:Y_{s}=0\}$. If $k$ is a locally bouded predictable process.Then,
$$k_{\gamma^{'}_{t}}Y_{t}=k_{0}Y_{0}+{\int_{0}^{t}{k_{\gamma^{'}_{s}}dY_{s}}}.$$
\end{prop}

Now, we shall recall the balayage formula in the progressive case.
\begin{prop}[\textbf{Ouknine and Bouhadou \cite{siam}} ]\label{balpro}
Let $Y$ be a continuous semimartingale and $\gamma^{'}_{t}=\sup\{s\leq t:Y_{s}=0\}$. If $k$ is a bounded progressive process ${^{p}k_{\cdot}}$ denotes its predictable projection.Then,
$$k_{\gamma^{'}_{t}}Y_{t}=k_{0}Y_{0}+{\int_{0}^{t}{^{p}k_{\gamma^{'}_{s}}dY_{s}}+R_{t}},$$
where $R$ is a process of bounded variation, adapted, continuous such that $dR_{t}$ is carried by the set $\{Y_{s}=0\}$.
\end{prop}

Proposition \ref{balpro} is a power and interesting tool. However, the fact to know nothing about the form of the process $R$ can be limiting. The utilisation of processes $Z^{\alpha}$ and $\mathcal{Z}^{\alpha}$ in this paper is capital. Bouhadou and Ouknin \cite{siam} have identified the process $R$ of Proposition \ref{balpro} when the progressive process $k$ is equal to $Z^{\alpha}$ or to $\mathcal{Z}^{\alpha}$. We recall these results in the following.  

\begin{prop}[\textbf{Ouknine and Bouhadou \cite{siam}}]
Let $Y$ be a continuous semimartingale and $Z^{\alpha}$, the process defined in \eqref{zalpha}. Then,
$$Z^{\alpha}_{t}Y_{t}=\int_{0}^{t}{Z^{\alpha}_{s}dY_{s}}+(2\alpha-1)L_{t}^{0}(Z^{\alpha}Y),$$
where, $L_{\cdot}^{0}(Z^{\alpha}Y)$ is the local time of the semimartingale $Z^{\alpha}Y$.
\end{prop}

\begin{prop}[\textbf{Ouknine and Bouhadou \cite{siam}}]
Let $Y$ be a continuous semimartingale and $Z^{\alpha}$, the process defined in \eqref{Zalpha}. Then,
$$\mathcal{Z}^{\alpha}_{t}Y_{t}=\int_{0}^{t}{\mathcal{Z}^{\alpha}_{s}dY_{s}}+\int_{0}^{t}{(2\alpha(s)-1)dL_{s}^{0}(\mathcal{Z}^{\alpha}Y)},$$
where, $L_{\cdot}^{0}(\mathcal{Z}^{\alpha}Y)$ is the local time of the semimartingale $\mathcal{Z}^{\alpha}Y$.
\end{prop}

\section{Martingales for signed measures}\label{s2} 

Throughout the rest of this paper, we shall consider a measure space $(\Omega, \mathcal{F}_{\infty}, \Qv)$, where $\Qv$ is a bounded signed measure. Let $\P$ be a probability measure on $\mathcal{F}_{\infty}$ such that $\Qv\ll\P$ and $(\mathcal{F}_{t})_{t\geq0}$ be a right continuous filtration completed with respect to $\P$ such that $\mathcal{F}_{\infty}=\vee_{t}{\mathcal{F}_{t}}$. 

This section is dedicated to the presentation of our contributions on the study of martingales for signed measures that  Ruiz de Chavez has introduced for the first time in \cite{chav}.

\subsection{Characterization results}

In this subsection, we present our main results on the study of  martingales for signed measures. We start by recalling how Ruiz de Chavez has defined a martingale with respect to a signed measure $\Qv$.

\begin{defn}\label{chav}
We say that an $(\mathcal{F}_{t})_{t\geq0}$- adapted process $M$ is a $(\Qv,\P)$- martingale if:
\begin{enumerate}
	\item $M$ is a $\P$- semimartingale.
	\item $MD$ is a $\P$- martingale.
\end{enumerate}
$M$ is a $(\Qv,\P)$- uniformly integrable  martingale (respectively local martingale) if $DM$ is a $\P$- uniformly integrable  martingale (respectively, local martingale).
\end{defn}

Remark that thanks to condition 1. of the above definition, we can use tools of the usual stochastic calculus to manipulate $(\Qv,\P)$- martingales. Furthermore, any  $(\Qv,\P)$- martingale $M$ admits the following decomposition: $M=m+v$, where $m$ is a $\P-$ martingale and $v$ is a finite variation process.

The following proposition is the main result of the present work. It gives a characterization of $(\Qv,\P)-$ martingales.

\begin{prop}\label{def}
Let $M=m+v$ be a $\P-$ semimartingale. The following hold
$$M\text{ is a } (\Qv,\P)-\text{ local martingale}\Leftrightarrow \forall t\geq0\text{, }\int_{0}^{t}{D_{s}dv_{s}}+\langle M,D\rangle_{t}=0.$$
\end{prop}
\begin{proof}
An application of an integration by parts gives the following
$$D_{t}M_{t}=\int_{0}^{t}{M_{s}dD_{s}}+\int_{0}^{t}{D_{s}dM_{s}}+\langle M,D\rangle_{t}.$$
That is,
$$D_{t}M_{t}=\left[\int_{0}^{t}{M_{s}dD_{s}}+\int_{0}^{t}{D_{s}dm_{s}}\right]+\int_{0}^{t}{D_{s}dv_{s}}+\langle M,D\rangle_{t}.$$
This completes the proof.
\end{proof}

Now, as an application of Proposition \ref{def}, we have the following interesting corollaries:

\begin{coro}\label{car}
Let $M=m+v$ be a $(\Qv,\P)-$ martingale such that $\langle M,D\rangle=0$. Then, $dv_{t}$ is carried by $H$.
\end{coro}
\begin{proof}
According to Proposition \ref{def}, we have
$$\int_{0}^{t}{D_{s}dv_{s}}+\langle M,D\rangle_{t}=0.$$
Hence, $\int_{0}^{t}{D_{s}dv_{s}}=0$ since $\langle M,D\rangle_{t}=0$.
Consequently, $dv_{t}$ is carried by $H$.
\end{proof}

\begin{coro}
Let $M=m+v$ be a $(\Qv,\P)-$ martingale such that $dv_{t}$ is carried by $H$. Then, the $\P-$ martingales $m$ and $D$ are orthogonal.
\end{coro}
\begin{proof}
According to Proposition \ref{def}, we have
$$\int_{0}^{t}{D_{s}dv_{s}}+\langle M,D\rangle_{t}=0.$$
Since $dv_{t}$ is carried by $H$, it entails that
 $$\int_{0}^{t}{D_{s}dv_{s}}=0.$$
Consequently, $\langle M,D\rangle_{t}=0$ since $\langle M,D\rangle_{t}=\langle m,D\rangle_{t}$.
This proves the result.
\end{proof}

\begin{coro}
Let $M=m+v$ be a $(\Qv,\P)-$ martingale such that $\langle M,D\rangle=0$. Then, $(M_{t}-M_{\gamma_{t}}:t\geq0)$ is a $\P-$ martingale null on $H$.
\end{coro}
\begin{proof}
One has from Proposition \ref{def} and Corollary \ref{car} that $dv_{t}$ is carried by $H$. Hence, 
$$\forall t\geq0,\text{ }v_{\gamma_{t}}=v_{t}.$$
Consequently,
$$\forall t\geq0,\text{ }M_{t}-M_{\gamma_{t}}=m_{t}-m_{\gamma_{t}}.$$
This completes the proof.
\end{proof}

\begin{coro}\label{car4}
Let $M=m+v$ be a $(\Qv,\P)-$ martingale vanishing on $H$ such that $\langle M,D\rangle=0$. Then, $M$ is a $\P-$ martingale.
\end{coro}

The following theorem characterizes the $(\Qv, \P)$- martingales whose all zeros are contained in $ H $.

\begin{theorem}\label{*mart}
Let $M=m+v$ be a continuous $\P-$ semi-martingale such that $\{t\geq0: M_{t}=0\}\subset H$. Then, the following are equivalent:
\begin{enumerate}
	\item $M$ is a $(\Qv,\P)-$ local martingale.
	\item $|M|$ is a $(\Qv,\P)-$ local martingale.
\end{enumerate}
\end{theorem}

\begin{proof}
$(1)\Rightarrow(2)$ An application of Tanaka's formula gives
$$|M_{t}|=\int_{0}^{t}{{\rm sgn}(M_{s})dM_{s}}+L_{t}^{0}(M).$$
It follows from an integration by parts,
$$D_{t}|M_{t}|=\int_{0}^{t}{|M_{s}|dD_{s}}+\int_{0}^{t}{D_{s}d|M_{s}|}+\langle D,|M|\rangle_{t}.$$
That is,
$$D_{t}|M_{t}|=\int_{0}^{t}{|M_{s}|dD_{s}}+\int_{0}^{t}{D_{s}{\rm sgn}(M_{s})dM_{s}}+\int_{0}^{t}{D_{s}dL_{s}^{0}(M)}+\int_{0}^{t}{{\rm sgn}(M_{s})d\langle D,M\rangle_{s}}.$$
But,$ dL_{t}^{0}(M)$ is carried by $H$ since $\{t\geq0: X_{t}=0\}\subset H$. Then, it follows that
$$\int_{0}^{t}{D_{s}dL_{s}^{0}(M)}=0.$$
Hence,
$$D_{t}|M_{t}|=\int_{0}^{t}{|M_{s}|dD_{s}}+\int_{0}^{t}{D_{s}{\rm sgn}(M_{s})dM_{s}}+\int_{0}^{t}{{\rm sgn}(M_{s})d\langle D,M\rangle_{s}}$$
$$\hspace{4cm}=\left(\int_{0}^{t}{|M_{s}|dD_{s}}+\int_{0}^{t}{D_{s}{\rm sgn}(M_{s})dm_{s}}\right)+\int_{0}^{t}{{\rm sgn}(M_{s})(D_{s}dv_{s}+d\langle D,M\rangle_{s})}.$$
But we know from Proposition \ref{def} that 
$$\int_{0}^{t}{D_{s}dv_{s}}+\langle D,M\rangle_{t}=0$$
since $M$ is a $(\Qv,\P)-$ local martingale. Therefore,
$$D_{t}|M_{t}|=\left(\int_{0}^{t}{|M_{s}|dD_{s}}+\int_{0}^{t}{D_{s}{\rm sgn}(M_{s})dm_{s}}\right).$$
This is, $D|M|$ is $\P-$ local martingale. Consequently, $|M|$ is a $(\Qv,\P)-$ local martingale. 
$$ $$
$(2)\Rightarrow(1)$ Let us set 
 $$K_{t}=\lim_{s\searrow t}\inf{\left(1_{\{M_{s}>0\}}-1_{\{M_{s}<0\}}\right)}.$$
 $K$ is a progressive and bounded process and $^{p}K_{\cdot}$ denotes its predictable projection. We know that
$$M_{t}=K_{\gamma_{t}}|M_{t}|.$$
And we have from the balayage formula in the progressive case that
 $$M_{t}=K_{\gamma_{t}}|M_{t}|=\int_{0}^{t}{^{p}(K_{\gamma_{s}})d|M_{s}|}+R_{t}$$
 where $R_{t}$ is a bounded variations process and $dR_{t}$ is carried by $\{t\geq0: M_{t}=0\}\subset H$.
It follows from an integration by parts
$$D_{t}M_{t}=\int_{0}^{t}{M_{s}dD_{s}}+\int_{0}^{t}{D_{s}^{p}(K_{\gamma_{s}})d|M_{s}|}+\int_{0}^{t}{D_{s}dR_{s}}+\int_{0}^{t}{^{p}(K_{\gamma_{s}})d\langle D, |M|\rangle_{s}}.$$
But, we have $\int_{0}^{t}{D_{s}dR_{s}}=0$ because $dR_{t}$ is carried by $H$. Hence, if we denote $|M|=m^{'}+v^{'}$, one has
$$D_{t}M_{t}=\left(\int_{0}^{t}{M_{s}dD_{s}}+\int_{0}^{t}{D_{s}^{p}(K_{\gamma_{s}})dm^{'}_{s}}\right)+\int_{0}^{t}{^{p}(K_{\gamma_{s}})(D_{s}dv^{'}_{s}+d\langle D, |M|\rangle_{s})}.$$
We know that
$$\int_{0}^{t}{D_{s}dv^{'}_{s}}+\langle D, |M|\rangle_{t}=0$$
since $|M|$ is a $(\Qv,\P)-$ local martingale. Then, $M$ is a $(\Qv,\P)-$ local martingale. This completes the proof.
\end{proof}

\subsection{Balayage of martingales for signed measures}
\begin{prop}\label{p1}
Let $M=m+v$ be a $(\Qv,\P)$- local martingale. Hence, for every locally bounded predictable process $k$, the process $k_{\gamma_{\cdot}}M$ is also a $(\Qv,\P)$- local martingale. And it decomposes as
\begin{equation}\label{eq1}
	k_{\gamma_{t}}M_{t}=k_{\gamma_{0}}(M_{0}-M_{\gamma_{0}})+k_{\gt}M_{\gt}+\int_{0}^{t}{k_{\gs}d(M_{s}-M_{\gs})}.
\end{equation}
\end{prop}
\begin{proof}
Since $DM$ vanishes on $H$, it follows by an application of balayage formula that: $\forall t\geq0$,
$$k_{\gt}M_{t}D_{t}=k_{\gamma_{0}}M_{0}D_{0}+\int_{0}^{t}{k_{\gs}d(M_{s}D_{s})}.$$
But $DM$ is by definition, a $\P$- locale martingale. This entails that $(k_{\gt}M_{t}D_{t};t\geq0)$ is a $\P$- local martingale. That is, $(k_{\gt}M_{t};t\geq0)$ is a $(\Qv,\P)$- local martingale. Now, remark that the $\P$- semi-martingale $(M_{t}-M_{\gt})_{t\geq0}$ vanishes on $H$. Then, we obtain from balayage formula the following
$$k_{\gamma_{t}}(M_{t}-M_{\gt})=k_{\gamma_{0}}(M_{0}-M_{\gamma_{0}})+\int_{0}^{t}{k_{\gs}d(M_{s}-M_{\gs})}.$$
This completes the proof.
\end{proof}

\begin{coro}\label{c1}
Let $M=m+v$ be a $(\Qv,\P)$- local martingale null on $H$. Hence, for every locally bounded predictable process $k$, the process $k_{\gamma_{\cdot}}M$ is also a $(\Qv,\P)$- local martingale null on $H$ and satisfying the following decomposition:
\begin{equation}\label{eq2}
	k_{\gamma_{t}}M_{t}=k_{\gamma_{0}}M_{0}+\int_{0}^{t}{k_{\gs}dM_{s}}.
\end{equation}
\end{coro}
\begin{proof}
We have from Proposition \ref{p1} that $(k_{\gt}M_{t};t\geq0)$ is a $(\Qv,\P)$- local martingale. Furthermore, $\forall t\geq0$, $M_{\gt}=0$ since $M$ vanishes on $H$. Consequently, one obtain from \eqref{eq1} that:
$$k_{\gamma_{t}}M_{t}=k_{\gamma_{0}}M_{0}+\int_{0}^{t}{k_{\gs}dM_{s}}.$$
\end{proof}

\begin{coro}\label{c2}
Let $M=m+v$ be a $(\Qv,\P)$- local martingale null on $H$ such that $\langle M,D\rangle\equiv0$. Hence, for every locally bounded predictable process $k$, the process $k_{\gamma_{\cdot}}M$ is also a $(\Qv,\P)$- local martingale null on $H$ and its finite variation part $A$, satisfies: $dA_{t}$ is carried by $H$.
\end{coro}
\begin{proof}
We know from Corollary \ref{c1} that $(k_{\gt}M_{t};t\geq0)$ is a $(\Qv,\P)$- local martingale and $\forall t\geq0$,
$$k_{\gamma_{t}}M_{t}=k_{\gamma_{0}}M_{0}+\int_{0}^{t}{k_{\gs}dM_{s}}=k_{\gamma_{0}}M_{0}+\int_{0}^{t}{k_{\gs}dm_{s}}+\int_{0}^{t}{k_{\gs}dv_{s}}.$$
But, $k_{\gs}dv_{s}=k_{s}dv_{s}$ since $dv_{s}$ is carried by $H$. Hence,
$$k_{\gamma_{t}}M_{t}=k_{\gamma_{0}}M_{0}+\int_{0}^{t}{k_{\gs}dm_{s}}+\int_{0}^{t}{k_{s}dv_{s}}.$$
We can see that $dA_{t}=k_{t}dv_{t}$ is carried by $H$ because $\langle M,D\rangle\equiv0$. Furthermore, we have 
$$\langle D,K_{\g}M\rangle_{t}=\int_{0}^{t}{k_{\gs}d\langle D,M\rangle_{s}}=0.$$
This completes proof.
\end{proof}

\begin{coro}\label{c3}
Let $M=m+v$ be a $(\Qv,\P)$- local martingale  vanishing on $H$ such that $\langle M,D\rangle\equiv0$ and $f:\R\rightarrow\R$ a locally bounded Borel function. Denote $F(x)=\int_{0}^{x}{f(s)ds}$. Then, the process $f(v)M$ is again a $(\Qv,\P)$- local martingale with decomposition
\begin{equation}
	f(v_{t})M_{t}=f(v_{0})M_{0}+\int_{0}^{t}{f(v_{s})dm_{s}}+F(v_{t}).
\end{equation}
\end{coro}
\begin{proof}
According to Corollary \ref{c1}, $f(v_{\g})M$ is a $(\Qv,\P)$- local martingale and one has the following:
$$f(v_{\gt})M_{t}=f(v_{\gamma_{0}})M_{0}+\int_{0}^{t}{f(v_{\gs})dm_{s}}+\int_{0}^{t}{f(v_{\gs})dv_{s}}.$$
But we know from Corollary \ref{car} that $dv_{t}$ is carried by $H$ since $\langle M,D\rangle\equiv0$. Thus $\forall t\geq0$, $v_{\gt}=v_{t}$. Therefore one has,
$$f(v_{t})M_{t}=f(v_{0})M_{0}+\int_{0}^{t}{f(v_{s})dm_{s}}+F(v_{t}).$$
This completes the proof.
\end{proof}

\begin{coro}\label{c4}
Any positive sub-martingale which is a $(\Qv,\P)$- local martingale  vanishing on $H$ such that $\langle M,D\rangle\equiv0$ is a process of the class $(\Sigma)$.
\end{coro}
\begin{proof}
Let $M=m+v$ be a such $(\Qv,\P)$- local martingale. We have thanks to Corollary \ref{c3} that for every locally bounded Borel function $f:\R\rightarrow\R$, the process $(f(v_{t})M_{t}-F(v_{t}))_{t\geq0}$ is a $\P$- local martingale. Hence, we obtain the result  by applying Theorem 2.1 of \cite{nik}.
\end{proof}

Now, we shall use Theorem \ref{*mart} to derive an another result permitting to characterize the $(\Qv, \P)$- martingales whose all zeros are contained in $H$ using the process $Z^{\alpha}$ that we have defined in \eqref{alpha}.

\begin{theorem}\label{alpha}
Let $M$ be a continuous $\P-$ semi-martingale such that $\{t\geq0: M_{t}=0\}\subset H$. Let us set $M^{\alpha}=Z^{\alpha}M$. The following are equivalent:
\begin{enumerate}
	\item $M$ is a $(\Qv,\P)-$ local martingale.
	\item $\forall\alpha\in(0,1)$, $M^{\alpha}$ is a $(\Qv,\P)-$ local martingale.
	\item $\exists\alpha\in(0,1)$ such that $M^{\alpha}$ is a $(\Qv,\P)-$ local martingale.
\end{enumerate}
\end{theorem}

\begin{proof}
An application of Proposition 2.2 of \cite{siam} gives
$$M^{\alpha}_{t}=\int_{0}^{t}{Z^{\alpha}_{s}dM_{s}}+(2\alpha-1)L_{t}^{0}(M^{\alpha})\text{, }\forall \alpha\in(0,1).$$
Hence, one obtains from an integration by parts the following
$$D_{t}M^{\alpha}_{t}=\int_{0}^{t}{M^{\alpha}_{s}dD_{s}}+\int_{0}^{t}{D_{s}dM^{\alpha}_{s}}+\langle D,M^{\alpha}\rangle_{t}.$$
That is,
$$D_{t}M^{\alpha}_{t}=\int_{0}^{t}{M^{\alpha}_{s}dD_{s}}+\int_{0}^{t}{D_{s}Z^{\alpha}_{s}dM_{s}}+(2\alpha-1)\int_{0}^{t}{D_{s}dL^{0}_{s}(M^{\alpha})}+\int_{0}^{t}{Z^{\alpha}_{s}d\langle D,M\rangle_{s}}.$$
But, 
$$(2\alpha-1)\int_{0}^{t}{D_{s}dL^{0}_{s}(M^{\alpha})}=0$$
since $dL^{0}_{s}(M^{\alpha})$ is carried by $\{t\geq0: M_{t}=0\}\subset H$. Therefore,
$$D_{t}M^{\alpha}_{t}=\int_{0}^{t}{M^{\alpha}_{s}dD_{s}}+\int_{0}^{t}{D_{s}Z^{\alpha}_{s}dM_{s}}+\int_{0}^{t}{Z^{\alpha}_{s}d\langle D,M\rangle_{s}}.$$
Then, it follows that
$$D_{t}M^{\alpha}_{t}=\left(\int_{0}^{t}{M^{\alpha}_{s}dD_{s}}+\int_{0}^{t}{D_{s}Z^{\alpha}_{s}dm_{s}}\right)+\int_{0}^{t}{Z^{\alpha}_{s}(D_{s}dv_{s}+d\langle D,M\rangle_{s})}.$$
An application of Proposition \ref{def} implies
$$\int_{0}^{t}{Z^{\alpha}_{s}(D_{s}dv_{s}+d\langle D,M\rangle_{s})}=0.$$
Hence, $DM^{\alpha}$ is a $\P-$ local martingale. Consequently, $M^{\alpha}$ is a $(\Qv,\P)-$ local martingale.
$$ $$
$(2)\Rightarrow(3)$ If we consider that $\forall\alpha\in(0,1)$, $M^{\alpha}$ is a $(\Qv,\P)-$ local martingale. It follows in particular that $\exists\alpha\in(0,1)$ such that $M^{\alpha}$ is a $(\Qv,\P)-$ local martingale.
$$ $$
$(3)\Rightarrow(1)$ We have from Theorem \ref{*mart} that $|M^{\alpha}|$ is a $(\Qv,\P)-$ local martingale. But, $|M^{\alpha}|=|M|$ because $Z^{\alpha}_{t}\in\{-1,1\}$. Hence, $|M|$ is a $(\Qv,\P)-$ local martingale. An another application of Theorem \ref{*mart} entails that $M$ is a $(\Qv,\P)-$ local martingale. This completes the proof.
\end{proof}

\subsection{Brownian motion for signed measures}\label{s3}

Now, we are interested by Brownian motions defined under a signed measure. Ruiz de Chavez characterizes them as stochastic processes defined as follows:

\begin{defn}\label{brow}
A continuous process $M$ is a $(\Qv,\P)-$ Brownian motion if one of next conditions holds:
\begin{enumerate}
	\item $M$ and $(M_{t}^{2}-t)_{t\geq0}$ are $(\Qv,\P)-$ locale martingales.
	\item $M$ is a $(\Qv,\P)-$ locale martingale and $\langle M,M\rangle_{t}=t$ $a.s.$
\end{enumerate}
\end{defn}

We give an another way to characterize $(\Qv,\P)-$ Brownian motions in the following result.

\begin{prop}\label{pc1}
A continuous $\P-$ semi-martingale $M=m+v$ is a $(\Qv,\P)-$ Brownian motion if, and only if, $m$ is $\P-$ Brownian motion and $\forall t\geq0$,
$$\int_{0}^{t}{D_{s}dv_{s}}+\langle D,M\rangle_{t}=0.$$
\end{prop}

\begin{proof}
$\Rightarrow)$ We have by assumptions that $m$ is a $\P-$ local martingale. Furthermore,
$$\langle m,m\rangle_{t}=\langle M,M\rangle_{t}=t.$$
Then, $m$ is a $\P-$ Brownian motion. Since $M$ is a $(\Qv,\P)-$ local martingale, we have from Proposition \ref{def} that 
$$\int_{0}^{t}{D_{s}dv_{s}}+\langle D,M\rangle_{t}=0.$$
$\Leftarrow)$ From Proposition \ref{def}, $M$  is a $(\Qv,\P)-$ local martingale. Furthermore, $\forall t\geq0$, $\langle M,M\rangle_{t}=\langle m,m\rangle_{t}=t$ since $m$ is a $\P-$ Brownian motion. Consequently, $M$ is a $(\Qv,\P)-$ Brownian motion.
\end{proof}

In the following, we give a corollary of Theorem \ref{*mart} which shows that the absolute value of some $(\Qv,\P)-$ Brownian motions is again a $(\Qv,\P)-$ Brownian motion.

\begin{coro}\label{*}
Let $M=m+v$ be a continuous $\P-$ semi-martingale such that $\{t\geq0: M_{t}=0\}\subset H$. Then, the following are equivalent:
\begin{enumerate}
	\item $M$ is a $(\Qv,\P)-$ Brownian motion.
	\item $|M|$ is a $(\Qv,\P)-$ Brownian motion.
\end{enumerate}
\end{coro}

\begin{proof}
We can obviously see from Theorem \ref{*mart} that $M$ is a $(\Qv,\P)-$ local martingale if, and only if, $|M|$ is a $(\Qv,\P)-$ local martingale. Furthermore,
$$\langle |M|,|M|\rangle=\langle M,M\rangle.$$
This completes the proof.
\end{proof}

Now, we give an another series of corollaries of Theorem \ref{alpha}.

\begin{coro}\label{za}
Let $M$ be a continuous $\P-$ semi-martingale such that $\{t\geq0: M_{t}=0\}=H$. Let us set $M^{\alpha}=Z^{\alpha}M$. The following are equivalent:
\begin{enumerate}
	\item $M$ is a $(\Qv,\P)-$ Brownian motion.
	\item $\forall\alpha\in(0,1)$, $M^{\alpha}$ is a $(\Qv,\P)-$ Brownian motion.
	\item $\exists\alpha\in(0,1)$ such that $M^{\alpha}$ is a $(\Qv,\P)-$ Brownian motion.
\end{enumerate}
\end{coro}

\begin{proof}
By using Theorem \ref{alpha}, we have that the following assertions are equivalent 
\begin{enumerate}
	\item $M$ is a $(\Qv,\P)-$ local martingale.
	\item $\forall\alpha\in(0,1)$, $M^{\alpha}$ is a $(\Qv,\P)-$ local martingale.
	\item $\exists\alpha\in(0,1)$ such that $M^{\alpha}$ is a $(\Qv,\P)-$ local martingale.
\end{enumerate}
Furthermore, 
$$\langle M^{\alpha},M^{\alpha}\rangle=\langle M, M\rangle.$$
This completes the proof.
\end{proof}

\begin{coro}\label{cbrow}
Let $M=m+v$ be a continuous $\P-$ semi-martingale such that $\{t\geq0: M_{t}=0\}=H$ and $\langle D,M\rangle=0$. Let us set $M^{\alpha}=Z^{\alpha}M$. The following are equivalent:
\begin{enumerate}
	\item $M$ is a $(\Qv,\P)-$ Brownian motion.
	\item $\exists\alpha\in(0,1)$ such that $M^{\alpha}$ is a $\P-$ Brownian motion.
\end{enumerate}
\end{coro}

\begin{proof}
$(1)\Rightarrow(2)$ Thanks to Corollary \ref{za}, $M^{\alpha}$ is a $(\Qv,\P)$- Brownian motion, $\forall\alpha\in(0,1)$. Hence, 
$$\langle M^{\alpha},M^{\alpha}\rangle=t.$$
Moreover,
$$M^{\alpha}_{t}=\int_{0}^{t}{Z^{\alpha}_{s}dM_{s}}+(2\alpha-1)L_{t}^{0}(M^{\alpha})\text{, }\forall \alpha\in(0,1).$$
By putting $\alpha=\frac{1}{2}$, it follows that
$$M^{\alpha}_{t}=\int_{0}^{t}{Z^{\alpha}_{s}dM_{s}}=\int_{0}^{t}{Z^{\alpha}_{s}dm_{s}}+\int_{0}^{t}{Z^{\alpha}_{s}dv_{s}}.$$
But,
$$\int_{0}^{t}{Z^{\alpha}_{s}dv_{s}}=0$$
since $\langle D,M\rangle=0$  and $Z^{\alpha}$ is defined on the complementary of $\{t\geq0: M_{t}=0\}=H$. Then,
$$M^{\alpha}_{t}=\int_{0}^{t}{Z^{\alpha}_{s}dm_{s}}.$$
That is, $M^{\alpha}$ is a $\P-$ local martingale. Consequently, it is a $\P-$ Brownian motion.
$$ $$
$(2)\Rightarrow(1)$ Now, let us assume that $\exists \alpha\in(0,1)$ such that $M^{\alpha}$ is a $\P-$ Brownian motion. Thus, 
$$\langle M^{\alpha},M^{\alpha}\rangle_{t}=t.$$ Furthermore, One has from an integration by parts 
$$D_{t}M^{\alpha}_{t}=\int_{0}^{t}{M^{\alpha}dD_{s}}+\int_{0}^{t}{D_{s}dM^{\alpha}}+\langle D,M^{\alpha}\rangle_{t}.$$
But,
$$\langle D,M^{\alpha}\rangle_{t}=\int_{0}^{t}{Z^{\alpha}_{s}d\langle D,M\rangle_{s}}=0$$
since $\langle D,M\rangle=0$.
Therefore, $M^{\alpha}$ is a $(\Qv,\P)-$ local martingale. Consequently, it is a $(\Qv,\P)-$ Brownian motion. Since $|M|=|M^{\alpha}|$, it follows after applying of Corollary \ref{*} that $M$ is a $(\Qv,\P)-$ Brownian motion.
\end{proof}

Gilat proved in \cite{gila} that every non negative sub-martingale $Y$ is equal in law to the absolute value of a Martingale $M$. His construction, however, did not shed any light on the nature of this martingale. Bouhadou and Ouknine have proposed a construction of $M$  in the case where $X$ is of class $(\Sigma)$. In what follows, we prove that the absolute value of some $(\Qv,\P)$- Brownian motions $M$, is equal to the absolute value of a $\P$- Brownian motion $B$.

\begin{coro}\label{absbrow}
Let $M=m+v$ be a continuous $(\Qv,\P)-$ Brownian motion such that $\{t\geq0: M_{t}=0\}=H$ and $\langle D,M\rangle=0$. Then,
there exists a $\P-$ Brownian motion $B$ such that $|M|=|B|$.
\end{coro}

\begin{proof}
According to Corollary \ref{cbrow}, $\exists\alpha\in(0,1)$ such that $B=Z^{\alpha}M$ is a $\P-$ Brownian motion. But, we have
$$|B|=|Z^{\alpha}M|=|M|.$$
This completes the proof.
\end{proof}

\begin{rem}
A  particular case deduced from Corollary \ref{absbrow}, is for every positive continuous $(\Qv,\P)-$ Brownian motion satisfying assumptions of Corollary \ref{absbrow}, we can found a $\P$- Brownian motion $B$ such that $M=|B|$.
\end{rem}

\subsection{Optional representation formula for \texorpdfstring{$(\Qv,\P)$}{} martingales}\label{s22}

The results in this subsection are inspired by a representation formula for relative martingales by Azéma and Yor \cite{1}. We consider these results as extensions of Doob's optional representation formula for the uniformly integrable $(\Qv,\P)$- martingales.

\begin{theorem}\label{opt}
Let $M$ be a uniformly integrable $(\Qv,\P)-$ martingale  with respect to a filtration $(\mathcal{F}_{t})_{t\geq0}$ such that $M_{\overline{g}}=0$ a.s. Hence, there exists a random variable $M_{\infty}$ such that $$M_{\infty}=\lim_{t\to+\infty}{M_{t}}$$
and for every stopping time $T<\infty$,
$$M_{T}-M_{\gamma_{T}}=E\left[M_{\infty}1_{\{\overline{g}<T\}}|\mathcal{F}_{T}\right].$$ 
\end{theorem}
\begin{proof}
Since $DM$ is a uniformly integrable $\P-$ martingale which vanishes on $H$ with respect to the filtration $(\mathcal{F}_{t})_{t\geq0}$. We obtain from quotient theorem that $({\rm sgn}(D_{t+\overline{g}})M_{t+\overline{g}}:t>0)$ is a uniformly integrable $\P^{'}-$ martingale with respect to the filtration $(\mathcal{G}_{t+\overline{g}})_{t>0}$. But, $D$ is a continuous process and $g=\sup\{t\geq0:D_{t}=0\}$. Then, $\forall t>0$, ${\rm sgn}(D_{t+\overline{g}})=cste\in\{-1,1\}$. Hence, $(M_{t+\overline{g}}:t>0)$ is a uniformly integrable $\P^{'}-$ martingale with respect to the filtration $(\mathcal{G}_{t+\overline{g}})_{t>0}$. Therefore, there exists a random variable $M_{\infty}$ such that $$M_{\infty}=\lim_{t\to+\infty}{M_{t+\overline{g}}}$$
and for every stopping time $T<\infty$,
$$M_{T+\overline{g}}-M_{\overline{g}}=E\left[M_{\infty}|\mathcal{G}_{T+\overline{g}}\right].$$ 
That is,
$$M_{T+\overline{g}}=E\left[M_{\infty}|\mathcal{G}_{T+\overline{g}}\right]$$
since $M_{\overline{g}}=0$. 
Thus,
$$\rho(M_{\cdot+\overline{g}})_{T}=\rho\left(E\left[M_{\infty}|\mathcal{G}_{\cdot+\overline{g}}\right]\right)_{T}.$$ 
Let us set $Z_{t}=M_{t}-M_{\gamma_{t}}$. $Z$ vanishes on $H$ and $\forall t>0$, $Z_{t+\overline{g}}=M_{t+\overline{g}}$. Then it entails from the uniqueness of Theorem \ref{rho} that
$$Z_{t}=\rho(M_{\cdot+\overline{g}})_{t}.$$
That is, 
$$M_{T}-M_{\gamma_{T}}=\rho\left(E\left[M_{\infty}|\mathcal{G}_{\cdot+\overline{g}}\right]\right)_{T}.$$ 
Now, we set 
$$Y_{t}=E\left[M_{\infty}1_{\{\overline{g}<t\}}|\mathcal{F}_{t}\right].$$
$Y$ vanishes on $H$ and $\forall t>0$,
$$Y_{t+\overline{g}}=E\left[M_{\infty}|\mathcal{F}_{t+\overline{g}}\right].$$
But according to Lemma \ref{j80}, we have for every stopping time $T$ that $\mathcal{F}_{T}=\mathcal{G}_{T}$ on $\{\overline{g}<T\}$. Then, $\mathcal{F}_{t+\overline{g}}=\mathcal{G}_{t+\overline{g}}$, $\forall t>0$. Consequently,
$$Y_{t+\overline{g}}=E\left[M_{\infty}|\mathcal{G}_{t+\overline{g}}\right].$$
It follows from uniqueness of Theorem \ref{rho} that 
$$Y_{T}=\rho\left(E\left[M_{\infty}|\mathcal{G}_{\cdot+\overline{g}}\right]\right)_{T}.$$
That is,
$$M_{T}-M_{\gamma_{T}}=E\left[M_{\infty}1_{\{\overline{g}<T\}}|\mathcal{F}_{T}\right].$$ 
\end{proof}

Now, we shall give some corollaries of Theorem \ref{opt}.

\begin{coro}\label{c5}
Let $M$ be a uniformly integrable $\left((\Qv,\P),(\mathcal{F}_{t})_{t\geq0}\right)-$ martingale vanishing on $H$. Then, for every stopping time $0<T<\infty$,
$$M_{T}=E\left[M_{\infty}1_{\{\overline{g}<T\}}|\mathcal{F}_{T}\right].$$
\end{coro}

\begin{coro}\label{c6}
Let $M=m+v$ be a uniformly integrable $\left((\Qv,\P),(\mathcal{F}_{t})_{t\geq0}\right)$- martingale such that $\langle D,M\rangle\equiv0$ and $M_{\overline{g}}=0$. Then, for any locally bounded Borel function $f$ and for every stopping time $0<T<\infty$,
$$f(v_{T})(M_{T}-M_{\gamma_{T}})=E\left[f(v_{\infty})M_{\infty}1_{\{\overline{g}<T\}}|\mathcal{F}_{T}\right].$$
\end{coro}
\begin{proof}
We have from Corollary \ref{c5} that for every stopping time $0<T<\infty$,
$$M_{T}-M_{\gamma_{T}}=E\left[M_{\infty}1_{\{\overline{g}<T\}}|\mathcal{F}_{T}\right].$$
Then,
$$f(v_{T})(M_{T}-M_{\gamma_{T}})=E\left[f(v_{T})M_{\infty}1_{\{\overline{g}<T\}}|\mathcal{F}_{T}\right].$$
Furthermore, we can see from Corollary \ref{car} that $dv_{t}$ is carried by $H$ since $\langle D,M\rangle\equiv0$. 
That is, 
$$f(v_{T})1_{\{\overline{g}<T\}}=f(v_{\infty})1_{\{\overline{g}<T\}}.$$
Hence,
$$f(v_{T})(M_{T}-M_{\gamma_{T}})=E\left[f(v_{\infty})M_{\infty}1_{\{\overline{g}<T\}}|\mathcal{F}_{T}\right].$$
\end{proof}

\begin{coro}
Let $M=m+v$ be a uniformly integrable $\left((\Qv,\P),(\mathcal{F}_{t})_{t\geq0}\right)-$ martingale vanishing on $H$ such that $\langle D,M\rangle\equiv0$. Then, for any locally bounded Borel function $f$ and for every stopping time $0<T<\infty$,
$$f(v_{T})M_{T}=E\left[f(v_{\infty})M_{\infty}1_{\{\overline{g}<T\}}|\mathcal{F}_{T}\right].$$
\end{coro}

\begin{proof}
An application of Corollary \ref{c6} gives for every stopping time $0<T<\infty$,
$$f(v_{T})(M_{T}-M_{\gamma_{T}})=E\left[f(v_{\infty})M_{\infty}1_{\{\overline{g}<T\}}|\mathcal{F}_{T}\right].$$
Then,
$$f(v_{T})M_{T}=E\left[f(v_{\infty})M_{\infty}1_{\{\overline{g}<T\}}|\mathcal{F}_{T}\right]$$
since $M$ vanishes on $H$.
\end{proof}

\section{Some contributions to the study of the class \texorpdfstring{$\Sigma(H)$}{sigma(H)}}\label{s4}

This section is devoted to the study of the class $\Sigma(H)$. Note that it is appeared for the first time in \cite{f} and studied in \cite{f,e,o}. Here, we provide new properties and we prove that the relevant class of stochastic processes can be characterized as we do it for $(\Qv,\P)$- martingales in Theorem \ref{*mart} and Theorem \ref{alpha}. So, we start by giving the following definition:

\begin{defn}
Let $X$ be a $\Pv$- semi-martingale, which decomposes as:
$$X_{t}=M_{t}+A_{t}.$$
We say that $X$ is of class $\Sigma(H)$, if:
\begin{enumerate}
	\item $M$ is a càdlàg $(\Qv,\Pv)$- local martingale with $M_{0}=0$;
	\item $A$ is a continuous finite variation process with $A_{0}=0$;
	\item the measure $(dA_{t})$ is carried by the set $\{t: X_{t}=0\}\cup H$.
\end{enumerate}
\end{defn}

\subsection{Some new properties} 

Now, we shall derive some new properties satisfied by stochastic processes of the class $\Sigma(H)$.

\begin{prop}
Let $X$ be a stochastic process of the class $\Sigma(H)$. Hence, $\left(\int_{0}^{t}{X_{s}dX_{s}}: t\geq0\right)$ is a \\$(\Qv,\P)-$ local martingale.
\end{prop}
\begin{proof}
Let us set $Z_{t}=\int_{0}^{t}{X_{s}dX_{s}}$ and $X=M+V\in\Sigma(H)$. An integration by part gives the following
$$\hspace{-2cm}D_{t}Z_{t}=\int_{0}^{t}{Z_{s}dD_{s}}+\int_{0}^{t}{D_{s}X_{s}dX_{s}}+\int_{0}^{t}{X_{s}d\langle D,X\rangle_{s}}$$
$$\hspace{1.5cm}=\int_{0}^{t}{Z_{s}dD_{s}}+\int_{0}^{t}{D_{s}X_{s}dM_{s}}+\int_{0}^{t}{D_{s}X_{s}dV_{s}}+\int_{0}^{t}{X_{s}d\langle D,M\rangle_{s}}.$$
But, 
$$\int_{0}^{t}{D_{s}X_{s}dV_{s}}=0$$
since $dV_{t}$ is carried by $\{t\geq0:D_{t}X_{t}=0\}$. That implies that
$$D_{t}Z_{t}=\int_{0}^{t}{Z_{s}dD_{s}}+\int_{0}^{t}{D_{s}X_{s}dM_{s}}+\int_{0}^{t}{X_{s}d\langle D,M\rangle_{s}}.$$
That is,
$$D_{t}Z_{t}=\int_{0}^{t}{Z_{s}dD_{s}}+\int_{0}^{t}{X_{s}d(D_{s}M_{s})}-\int_{0}^{t}{X_{s}M_{s}dD_{s}}$$
since $d(D_{s}M_{s})=D_{s}dM_{s}+M_{s}dD_{s}+d\langle D,M\rangle_{s}$. Hence, $DZ$ is a $\P-$ local martingale since $D$ and $MD$ are too. Consequently, $Z$ is a $(\Qv,\P)-$ local martingale.
\end{proof}

\begin{prop}\label{1}
 Let $X=M+V$ be a stochastic process of class $\Sigma(H)$. Then, for any bounded predictable process $K$, 
 $$(K_{g_{t}}\cdot X_{t})\in \Sigma(H)$$
where, $g_{t}=\sup\{s\leq t: X_{s}=0\}$.
\end{prop}
 \begin{proof}
 An application of the balayage formula gives
 $$K_{g_{t}}X_{t}=\int_{0}^{t}{K_{g_{s}}dX_{s}}=\int_{0}^{t}{K_{g_{s}}dM_{s}}+\int_{0}^{t}{K_{g_{s}}dV_{s}}$$
 and the assertion follows.
 \end{proof}

\begin{prop}\label{p2}
 Let $X=M+V$ be a stochastic process of class $\Sigma(H)$. Then, for any bounded predictable process $K$, 
 $$(K_{\gamma_{t}}\cdot X_{t})\in \Sigma(H)$$
where, $\gamma_{t}=\sup\{s\leq t: D_{s}X_{s}=0\}$.
\end{prop}
 \begin{proof}
 An application of the balayage formula gives
 $$\hspace{-4cm}K_{\gamma_{t}}D_{t}X_{t}=\int_{0}^{t}{K_{\gamma_{s}}d(D_{s}X_{s})}$$
$$\hspace{0.5cm}=\int_{0}^{t}{K_{\gamma_{s}}d(D_{s}M_{s})}+\int_{0}^{t}{K_{\gamma_{s}}d(D_{s}V_{s})}$$
$$\hspace{3cm}=\int_{0}^{t}{K_{\gamma_{s}}d(D_{s}M_{s})}+\int_{0}^{t}{K_{\gamma_{s}}V_{s}dD_{s}}+\int_{0}^{t}{K_{\gamma_{s}}D_{s}dV_{s}}.$$
Let us set 
$$Y_{t}=\int_{0}^{t}{K_{\gamma_{s}}dV_{s}}.$$
One has,
$$D_{t}Y_{t}=\int_{0}^{t}{Y_{s}dD_{s}}+\int_{0}^{t}{D_{s}K_{\gamma_{s}}dV_{s}}.$$
This entails,
$$\int_{0}^{t}{K_{\gamma_{s}}D_{s}dV_{s}}=D_{t}Y_{t}-\int_{0}^{t}{Y_{s}dD_{s}}.$$
Then,
$$K_{\gamma_{t}}D_{t}X_{t}=\int_{0}^{t}{K_{\gamma_{s}}d(D_{s}M_{s})}+\int_{0}^{t}{K_{\gamma_{s}}V_{s}dD_{s}}+D_{t}Y_{t}-\int_{0}^{t}{Y_{s}dD_{s}}.$$
That is,
$$D_{t}(K_{\gamma_{t}}X_{t}-Y_{t})=\int_{0}^{t}{K_{\gamma_{s}}d(D_{s}M_{s})}+\int_{0}^{t}{K_{\gamma_{s}}V_{s}dD_{s}}-\int_{0}^{t}{Y_{s}dD_{s}}.$$
Therefore, $(D_{t}K_{\gamma_{t}}X_{t}-D_{t}Y_{t};t\geq0)$ is a $\P-$ local martingale. Which means that 
$$M^{'}=\left(K_{\gamma_{t}}X_{t}-\int_{0}^{t}{K_{\gamma_{s}}dV_{s}};t\geq0\right)$$
 is a $(\Qv,\P)-$ local martingale. Consequently, $(K_{\gamma_{t}}\cdot X_{t})\in \Sigma(H)$.
 \end{proof}

\begin{coro}
Let $X=M+V$ be a stochastic process of class $\Sigma(H)$.  Then, for any locally bounded borel function $f$, $f(V)X\in \Sigma(H)$ and its finite variation part is $\int_{0}^{t}{f(V_{s})dV_{s}}$. 
\end{coro}

\subsection{New characterization results of the class \texorpdfstring{$\Sigma(H)$}{sigma(H)}}\label{6}

Now, we shall provide new ways to characterize stochastic processes of the class $\Sigma(H)$. The results of the present sub-section are inspired by Theorem \ref{*mart} and Theorem \ref{alpha}.

\begin{theorem}\label{abs}
Let $X$ be a continuous semi-martingale. Then, the following are equivalent:
\begin{enumerate}
	\item $X\in\Sigma(H)$
	\item $|X|\in\Sigma(H)$
\end{enumerate}
\end{theorem}
\begin{proof}
$(1)\Rightarrow(2)$ Assume that $X=M+V$ is a stochastic process of class $\Sigma(H)$. An application of Itô-Tanaka formula gives
$$|X_{t}|=\int_{0}^{t}{{\rm sgn}(X_{s})dX_{s}}+L_{t}^{0}(X).$$
By definition, $dL_{t}^{0}(X)$ is carried by $\{t\geq0: |X_{t}|=0\}$. Then, $dL_{t}^{0}(X)$ is carried by $\{t\geq0: D_{t}|X_{t}|=0\}$. Now, we shall show that the stochastic process defined by $Y_{t}=\int_{0}^{t}{{\rm sgn}(X_{s})dX_{s}}$ is a $(\Qv,\P)-$ local martingale. One has:
$$\hspace{-5.5cm}D_{t}Y_{t}=\int_{0}^{t}{D_{s}dY_{s}}+\int_{0}^{t}{Y_{s}dD_{s}}+\langle Y,D\rangle_{t}$$
$$\hspace{-1cm}=\int_{0}^{t}{{\rm sgn}(X_{s})D_{s}dX_{s}}+\int_{0}^{t}{Y_{s}dD_{s}}+\int_{0}^{t}{{\rm sgn}(X_{s})d\langle X,D\rangle_{s}}$$
$$\hspace{2.75cm}=\int_{0}^{t}{{\rm sgn}(X_{s})D_{s}dM_{s}}+\int_{0}^{t}{{\rm sgn}(X_{s})D_{s}dV_{s}}+\int_{0}^{t}{Y_{s}dD_{s}}+\int_{0}^{t}{{\rm sgn}(X_{s})d\langle M,D\rangle_{s}}.$$
But, $\int_{0}^{t}{{\rm sgn}(X_{s})D_{s}dV_{s}}=0$ since, $dV_{t}$ is carried by $\{t\geq0: D_{t}X_{t}=0\}=\{t\geq0: D_{t}{\rm sgn}(X_{t})=0\}$. 
Then, 
$$D_{t}Y_{t}=\int_{0}^{t}{{\rm sgn}(X_{s})D_{s}dM_{s}}+\int_{0}^{t}{Y_{s}dD_{s}}+\int_{0}^{t}{{\rm sgn}(X_{s})d\langle M,D\rangle_{s}}.$$
$$D_{t}Y_{t}=\int_{0}^{t}{{\rm sgn}(X_{s})D_{s}dM_{s}}+\int_{0}^{t}{{\rm sgn}(X_{s})M_{s}dD_{s}}-\int_{0}^{t}{{\rm sgn}(X_{s})M_{s}dD_{s}}+\int_{0}^{t}{Y_{s}dD_{s}}+\int_{0}^{t}{{\rm sgn}(X_{s})d\langle M,D\rangle_{s}}.$$
Thus,
$$D_{t}Y_{t}=\int_{0}^{t}{{\rm sgn}(X_{s})d(D_{s}M_{s})}+\int_{0}^{t}{(Y_{s}-{\rm sgn}(X_{s})M_{s})dD_{s}}.$$
Since $DM$ and $D$ are $\P-$ local martingales. Therefore, $DY$ is a $\P-$ local martingale. That is, $Y$ is a $(\Qv,\P)-$ local martingale. Consequently, $|X|\in\Sigma(H)$. 

$(2)\Rightarrow(1)$ Let us set 
 $$K_{t}=\lim_{s\searrow t}\inf{\left(1_{\{X_{s}>0\}}-1_{\{X_{s}<0\}}\right)}.$$
 $K$ is a progressive and bounded process and $^{p}K_{\cdot}$ denotes its predictable projection. We know that
$$X_{t}=K_{g_{t}}|X_{t}|.$$
And we have from the balayage formula in the progressive case that
 $$K_{g_{t}}|X_{t}|=\int_{0}^{t}{^{p}(K_{g_{s}})d|X_{s}|}+R_{t}$$
 where $R_{t}$ is a bounded variations process and $dR_{t}$ is carried by $$\{t\geq0: X_{t}=0\}=\{t\geq0: |X_{t}|=0\}\subset\{t\geq0: D_{t}K_{g_{t}}|X_{t}|=0\}.$$
 Furthermore, $|X|\in\Sigma(H)$. Hence, it can be written as $|X|=M+V$, where $M$ is a $(\Qv,\P)-$ local martingale and $V$ is a finite variation and continuous process such that $dV_{t}$ is carried by $\{t\geq0: D_{t}|X_{t}|=0\}$. Then, it follows that 
$$K_{g_{t}}|X_{t}|=\int_{0}^{t}{^{p}(K_{g_{s}})dM_{s}}+\int_{0}^{t}{^{p}(K_{g_{s}})dV_{s}}+R_{t}=\int_{0}^{t}{^{p}(K_{g_{s}})dM_{s}}+A_{t}$$
with, $A_{t}=\int_{0}^{t}{^{p}(K_{g_{s}})dV_{s}}+R_{t}$. But, $\left(\int_{0}^{t}{^{p}(K_{g_{s}})dM_{s}};t\geq0\right)$ is $(\Qv,\P)-$ local martingale since $M$ is too. Furthermore, $dA_{t}$ is carried by $\{t\geq0: D_{t}X_{t}=0\}$. Consequently, 
 $$K_{g_{t}}|X_{t}|\in\Sigma(H).$$
That is, $X\in\Sigma(H)$.
\end{proof}

\begin{lem}\label{l1}
Let $X$ be a process of class $\Sigma(H)$. Then, $\forall \alpha\in[0,1]$, $\left(\int_{0}^{t}{Z^{\alpha}_{s}dX_{s}};t\geq0\right)$ is a $(\Qv,\P)-$ local martingale.
\end{lem}
\begin{proof}
Consider $X=M+V\in\Sigma(H)$. Let us set
$$Y_{t}=\int_{0}^{t}{Z^{\alpha}_{s}dX_{s}}.$$
An application of integration by parts gives:
$$D_{t}Y_{t}=\int_{0}^{t}{Z^{\alpha}_{s}D_{s}dX_{s}}+\int_{0}^{t}{Y_{s}dD_{s}}+\int_{0}^{t}{Z^{\alpha}_{s}d\langle X,D \rangle_{s}}.$$
But,
$$\langle X,D \rangle_{s}=\langle M,D \rangle_{s}.$$
Hence, we have
$$D_{t}Y_{t}=\int_{0}^{t}{Z^{\alpha}_{s}D_{s}dX_{s}}+\int_{0}^{t}{Y_{s}dD_{s}}+\int_{0}^{t}{Z^{\alpha}_{s}d\langle M,D \rangle_{s}}.$$
That is,
$$D_{t}Y_{t}=\int_{0}^{t}{Z^{\alpha}_{s}D_{s}dM_{s}}+\int_{0}^{t}{Z^{\alpha}_{s}D_{s}dV_{s}}+\int_{0}^{t}{Y_{s}dD_{s}}+\int_{0}^{t}{Z^{\alpha}_{s}d\langle M,D \rangle_{s}}.$$
We can remark that
$$\int_{0}^{t}{Z^{\alpha}_{s}D_{s}dV_{s}}=0,$$
because $Z^{\alpha}$ vanishes on the same zero set as $X$ and $dV_{t}$ is carried by $\{t\geq0: D_{t}X_{t}=0\}$. Hence,
$$\hspace{-3.5cm}D_{t}Y_{t}=\int_{0}^{t}{Z^{\alpha}_{s}D_{s}dM_{s}}+\int_{0}^{t}{Y_{s}dD_{s}}+\int_{0}^{t}{Z^{\alpha}_{s}d\langle M,D \rangle_{s}}$$
$$\hspace{3.5cm}=\int_{0}^{t}{Z^{\alpha}_{s}D_{s}dM_{s}}+\left(\int_{0}^{t}{Z^{\alpha}_{s}M_{s}dD_{s}}-\int_{0}^{t}{Z^{\alpha}_{s}M_{s}dD_{s}}\right)+\int_{0}^{t}{Y_{s}dD_{s}}+\int_{0}^{t}{Z^{\alpha}_{s}d\langle M,D \rangle_{s}}.$$
Thus,
$$\hspace{-3.5cm}D_{t}Y_{t}=\int_{0}^{t}{Z^{\alpha}_{s}d(D_{s}M_{s})}+\int_{0}^{t}{(Y_{s}-Z^{\alpha}_{s}M_{s})dD_{s}}$$
Consequently, $DY$ is a $\P-$ local martingale, since $DM$ and $D$ are too. That is, $Y$ is a $(\Qv,\P)-$ local martingale.
\end{proof}
\begin{theorem}\label{z}
Let $X$ be a continuous semi-martingale. The following are equivalent:
\begin{enumerate}
	\item $X\in\Sigma(H)$.
	\item $\forall\alpha\in[0,1]$, $Z^{\alpha}X\in\Sigma(H)$.
	\item $\exists\alpha\in[0,1]$ such that $Z^{\alpha}X\in\Sigma(H)$.
\end{enumerate}
\end{theorem}
\begin{proof}
$1\Rightarrow2)$ Let $X=M+V$ be an element of the class $\Sigma(H)$. One has what follows
$$Z^{\alpha}_{t}X_{t}=\int_{0}^{t}{Z^{\alpha}_{s}dX_{s}}+(2\alpha-1)L_{t}^{0}(Z^{\alpha}X).$$
We can see from Lemma \ref{l1} that $\int_{0}^{t}{Z^{\alpha}_{s}dX_{s}}$ is a $(\Qv,\P)-$ local martingale. Furthermore, 
 $(2\alpha-1)dL_{t}^{0}(Z^{\alpha}X)$ is carried by $\{t\geq0; Z^{\alpha}_{t}X_{t}=0\}$ and $\{t\geq0; Z^{\alpha}_{t}X_{t}=0\}\subset\{t\geq0; Z^{\alpha}_{t}X_{t}D_{t}=0\}$. That is, $Z^{\alpha}X\in\Sigma(H)$.
$$ $$
$2\Rightarrow3)$ If we consider that $\forall\alpha\in[0,1]$, $Z^{\alpha}X\in\Sigma(H)$. It follows in particular that $\exists\alpha\in[0,1]$ such that $Z^{\alpha}X\in\Sigma(H)$.
$$ $$
$3\Rightarrow1)$ Now, assume that $\exists\alpha\in[0,1]$ such that $Z^{\alpha}X\in\Sigma(H)$. Then, according to Theorem \ref{abs}, $|Z^{\alpha}X|\in\Sigma(H)$. But $\forall t\geq0$, $Z^{\alpha}_{t}\in\{-1,1\}$. Therefore,  
$$|Z^{\alpha}X|=|X|.$$
Consequently, 
$$|X|\in\Sigma(H).$$
A new application of Theorem \ref{abs} entails that $X\in\Sigma(H)$. This completes the proof.
\end{proof}

Now, as an application of Theorem \ref{z}, we have the following corollary. It gives a new characterization martingale of the class $\Sigma(H)$.
\begin{coro}\label{cmart}
Let $X$ be a continuous semimartingale. Then, the following holds:
$$X\in\Sigma(H)\Leftrightarrow \exists\alpha\in[0,1] \text{ such that }Z^{\alpha}X \text{ is a }(\Qv,\P)-\text{ local martingale}.$$
\end{coro}
\begin{proof}
$\Rightarrow)$ We know that $\forall \alpha\in[0,1]$,
$$Z^{\alpha}_{t}X_{t}=\int_{0}^{t}{Z^{\alpha}_{s}dX_{s}}+(2\alpha-1)L_{t}^{0}(Z^{\alpha}X).$$
Hence, we obtain in particular for $\alpha=\frac{1}{2}$ that
$$Z^{\alpha}_{t}X_{t}=\int_{0}^{t}{Z^{\alpha}_{s}dX_{s}}.$$
 But according to Lemma \ref{l1}, $Z^{\alpha}X$ is a $(\Qv,\P)-$ local martingale.

$\Leftarrow)$ Now, if we assume that $\exists\alpha\in[0,1]$ such that $Z^{\alpha}X$ is a $(\Qv,\P)-$ local martingale. It follows that $Z^{\alpha}X\in\Sigma(H)$. Then, we obtain from Theorem \ref{z} that $X\in\Sigma(H)$.
\end{proof}

\begin{coro}\label{sia}
If $X$ is a non-negative and continuous process of the class $\Sigma(H)$, then there exists a $(\Qv,\P)-$ local martingale $M$ such that $X=|M|$.
\end{coro}

\section{Construction of solutions for skew Brownian motion}\label{s5}

The aim of the present section is to provide constructions of  solutions for skew Brownian motion equations by using stochastic processes that we have studied in Section \ref{s2} and Section \ref{s4}. So, we consist this section in two subsections.  In the first one, we give solutions for the time homogeneous skew Brownian motion equation: 
\begin{equation}\label{sbm}
	X_{t}=x+B_{t}+(2\alpha-1)L_{t}^{0}(X).
\end{equation}

In the second one, we shall provide solutions for the time inhomogeneous skew Brownian motion equation:
\begin{equation}\label{isbm}
	X_{t}=x+B_{t}+\int_{0}^{t}{(2\alpha(s)-1)dL_{s}^{0}(X)}.
\end{equation}

Note that the constructions we propose here are inspired by the one proposed by Bouhadou and Ouknine \cite{siam}. In fact, under some assumptions, the solution proposed in \cite{siam} is a particular case of our solutions.

\subsection{Construction of solutions for homogeneous skew Brownian motion}

\begin{theorem}
Let $M=m+v$ be a $(\Qv,\P)-$ Brownian motion null on $H$ such that $\langle D,M\rangle=0$. Then, the following hold:
\begin{enumerate}
	\item $M^{\alpha}=Z^{\alpha}M$ is a weak solution of \eqref{sbm}.
	\item $|M|^{\alpha}=Z^{\alpha}|M|$ is a weak solution of \eqref{sbm}.
\end{enumerate}
\end{theorem}
\begin{proof}
$(1)$ We obtain by applying Proposition 2.2 of \cite{siam} that
$$M^{\alpha}_{t}=\int_{0}^{t}{Z^{\alpha}_{s}dM_{s}}+(2\alpha-1)L_{t}^{0}(M^{\alpha}).$$
That is,
$$M^{\alpha}_{t}=\int_{0}^{t}{Z^{\alpha}_{s}dm_{s}}+\int_{0}^{t}{Z^{\alpha}_{s}dv_{s}}+(2\alpha-1)L_{t}^{0}(M^{\alpha}).$$
Since $\langle D,M\rangle=0$, it follows from Corollary \ref{car} that $dv_{t}$ is carried by $H$. Hence,
$$\int_{0}^{t}{Z^{\alpha}_{s}dv_{s}}=0$$
since $H\subset\{t\geq0: M_{t}=0\}$ and the process $Z^{\alpha}$ is defined on the complementary of the set $\{t\geq0: M_{t}=0\}$. Then,
$$M^{\alpha}_{t}=\int_{0}^{t}{Z^{\alpha}_{s}dm_{s}}+(2\alpha-1)L_{t}^{0}(M^{\alpha}).$$
It is obvious to see that $W=\left(\int_{0}^{t}{Z^{\alpha}_{s}dm_{s}:t\geq0}\right)$ is a $\P-$ local martingale since $m$ is too. Moreover,
$$\langle W,W\rangle_{t}=\int_{0}^{t}{(Z^{\alpha}_{s})^{2}d\langle m,m\rangle_{s}}=\langle m,m\rangle_{t}$$
since $Z^{\alpha}_{s}\in\{-1,1\}$. But, $\langle m,m\rangle_{t}=\langle M,M\rangle_{t}=t$. Consequently, $W$ is a $\P-$ Brownian motion, ensuring that $M^{\alpha}$ satisfies \eqref{sbm}.
$$ $$
(2) We have from Proposition 2.2 of \cite{siam} that
$$|M|^{\alpha}_{t}=\int_{0}^{t}{Z^{\alpha}_{s}d|M_{s}|}+(2\alpha-1)L_{t}^{0}(|M|^{\alpha}).$$
That is,
$$|M|^{\alpha}_{t}=\int_{0}^{t}{Z^{\alpha}_{s}{\rm sgn(M_{s})}dM_{s}}+\int_{0}^{t}{Z^{\alpha}_{s}dL^{0}_{s}(M)}+(2\alpha-1)L_{t}^{0}(|M|^{\alpha}).$$
But,
$$\int_{0}^{t}{Z^{\alpha}_{s}dL^{0}_{s}(M)}=0$$
since $dL_{t}^{0}(M)$ is carried by $\{t\geq0:M_{t}=0\}$ and $Z^{\alpha}_{s}$ is defined on the complementary of $\{t\geq0:M_{t}=0\}$. So, that entails that
$$|M|^{\alpha}_{t}=\int_{0}^{t}{Z^{\alpha}_{s}dW_{s}}+(2\alpha-1)L_{t}^{0}(|M|^{\alpha})$$
where, $W_{t}=\int_{0}^{t}{{\rm sgn}(M_{s})dM_{s}}$. But, we can see from Corollary \ref{car4} that $W$ is a $\P-$ Brownian motion. That implies that $\left(\int_{0}^{t}{Z^{\alpha}_{s}dW_{s}}:t\geq0\right)$ is a $\P-$ Brownian motion, ensuring that $|M|^{\alpha}$ satisfies \eqref{sbm}.
\end{proof}

In the next theorem, we construct solutions of skew Brownian motion equation with stochastic processes of the class  $\Sigma(H)$.

\begin{theorem}
Let $X=M+V$ be a continuous stochastic process of class $\Sigma(H)$ null on $H$ such that $\langle X,X\rangle_{t}=t$ and  $\langle X,D\rangle_{t}=0$. Then, the following hold:
\begin{enumerate}
	\item $X^{\alpha}=Z^{\alpha}X$ is a weak solution of \eqref{sbm}.
	\item $|X|^{\alpha}=Z^{\alpha}|X|$ is a weak solution of \eqref{sbm}.
\end{enumerate}
\end{theorem}
\begin{proof}
$(1)$ We have from Proposition 2.2 of \cite{siam} that 
$$X^{\alpha}_{t}=\int_{0}^{t}{Z^{\alpha}_{s}dX_{s}}+(2\alpha-1)L_{t}^{0}(X^{\alpha}).$$
That is,
$$X^{\alpha}_{t}=\int_{0}^{t}{Z^{\alpha}_{s}dM_{s}}+\int_{0}^{t}{Z^{\alpha}_{s}dV_{s}}+(2\alpha-1)L_{t}^{0}(X^{\alpha}).$$
We know that $dV_{t}$ is carried by $\{t\geq0: X_{t}D_{t}=0\}$. But, $\{t\geq0: X_{t}D_{t}=0\}=\{t\geq0: X_{t}=0\}$ since $X$ vanishes on $H$. Hence, it follows that
$$X^{\alpha}_{t}=\int_{0}^{t}{Z^{\alpha}_{s}dM_{s}}+(2\alpha-1)L_{t}^{0}(X^{\alpha}).$$
Moreover, the $(\Qv,\P)-$ local martingale $M=m+v$ satisfies $\langle M,D\rangle_{t}=\langle X,D\rangle_{t}=0$. Then, it follows from Corollary \ref{car} that $dv_{t}$ is carried by $H\subset \{t\geq0: X_{t}=0\}$. Therefore,
$$ \int_{0}^{t}{Z^{\alpha}_{s}dv_{s}}=0.$$
Thus,
$$X^{\alpha}_{t}=\int_{0}^{t}{Z^{\alpha}_{s}dm_{s}}+(2\alpha-1)L_{t}^{0}(X^{\alpha}).$$
It is obvious to see that $Y=\left(\int_{0}^{t}{Z^{\alpha}_{s}dm_{s}}: t\geq0\right)$ is a $\P-$ local martingale. Furthermore, 
$$\langle Y,Y\rangle_{t}=\langle m,m\rangle_{t}=\langle M,M\rangle_{t}=\langle X,X\rangle_{t}=t.$$
Consequently, $Y$ is a $\P-$ Brownian motion ensuring that $X^{\alpha}$ satisfies  \eqref{sbm}.
$$ $$
$(2)$ One has, 
$$|X|^{\alpha}_{t}=\int_{0}^{t}{Z^{\alpha}_{s}d|X_{s}|}+(2\alpha-1)L_{t}^{0}(|X|^{\alpha})$$
$$\hspace{5cm}=\int_{0}^{t}{Z^{\alpha}_{s}{\rm sgn}(X_{s})dX_{s}}+\int_{0}^{t}{Z^{\alpha}_{s}dL^{0}_{s}(X)}+(2\alpha-1)L_{t}^{0}(|X|^{\alpha}).$$
But,
$$\int_{0}^{t}{Z^{\alpha}_{s}dL^{0}_{s}(X)}=0$$
since $dL_{t}^{0}(X)$ is carried by $\{t\geq0: X_{t}=0\}$ and $Z^{\alpha}$ is defined on the complementary of $\{t\geq0: X_{t}=0\}$. Hence,
$$|X|^{\alpha}_{t}=\int_{0}^{t}{Z^{\alpha}_{s}{\rm sgn}(X_{s})dX_{s}}+(2\alpha-1)L_{t}^{0}(|X|^{\alpha})$$
That implies,
$$|X|^{\alpha}_{t}=\int_{0}^{t}{Z^{\alpha}_{s}{\rm sgn}(X_{s})dM_{s}}+\int_{0}^{t}{Z^{\alpha}_{s}{\rm sgn}(X_{s})dV_{s}}+(2\alpha-1)L_{t}^{0}(|X|^{\alpha}).$$
We have 
$$\int_{0}^{t}{Z^{\alpha}_{s}{\rm sgn}(X_{s})dV_{s}}=0$$
because $dV_{t}$ is carried by $\{t\geq0: X_{t}=0\}$. Thus,
$$|X|^{\alpha}_{t}=\int_{0}^{t}{Z^{\alpha}_{s}{\rm sgn}(X_{s})dM_{s}}+(2\alpha-1)L_{t}^{0}(|X|^{\alpha})$$
$$\hspace{4.5cm}=\int_{0}^{t}{Z^{\alpha}_{s}{\rm sgn}(X_{s})dm_{s}}+\int_{0}^{t}{Z^{\alpha}_{s}{\rm sgn}(X_{s})dv_{s}}+(2\alpha-1)L_{t}^{0}(|X|^{\alpha}).$$
One has, $\langle M,D\rangle_{t}=\langle X,D\rangle_{t}=0$. Hence, we obtain from Corollary \ref{car} that $dv_{t}$ is carried by $H\subset\{t\geq0: X_{t}=0\}$. Then,
$$\int_{0}^{t}{Z^{\alpha}_{s}{\rm sgn}(X_{s})dv_{s}}=0.$$
That is, 
$$|X|^{\alpha}_{t}=\int_{0}^{t}{Z^{\alpha}_{s}{\rm sgn}(X_{s})dm_{s}}+(2\alpha-1)L_{t}^{0}(|X|^{\alpha}).$$
$W=\left(\int_{0}^{t}{Z^{\alpha}_{s}{\rm sgn}(X_{s})dm_{s}}:t\geq0\right)$ is a $\P-$ local martingale and 
$$\langle W,W\rangle_{t}=\langle m,m\rangle_{t}=\langle M,M\rangle_{t}=\langle X,X\rangle_{t}=t.$$
Consequently, $W$ is a $\P-$ Brownian motion ensuring that $|X|^{\alpha}$ satisfies  \eqref{sbm}.
\end{proof}

\subsection{Construction of solutions for time inhomogeneous skew Brownian motion}

\begin{theorem}
Let $M$ be a $(\Qv,\P)-$ Brownian motion null on $H$ such that $\langle D,M\rangle=0$. Then, the following hold:
\begin{enumerate}
	\item $M^{\alpha}=\mathcal{Z}^{\alpha}M$ is a weak solution of \eqref{isbm}.
	\item $|M|^{\alpha}=\mathcal{Z}^{\alpha}|M|$ is a weak solution of \eqref{isbm}.
\end{enumerate}
\end{theorem}
\begin{proof}
$(1)$ We obtain by applying Proposition 2.4 of \cite{siam} that
$$M^{\alpha}_{t}=\int_{0}^{t}{\mathcal{Z}^{\alpha}_{s}dM_{s}}+\int_{0}^{t}{(2\alpha(s)-1)dL_{s}^{0}(M^{\alpha})}.$$
That is,
$$M^{\alpha}_{t}=\int_{0}^{t}{\mathcal{Z}^{\alpha}_{s}dm_{s}}+\int_{0}^{t}{\mathcal{Z}^{\alpha}_{s}dv_{s}}+\int_{0}^{t}{(2\alpha(s)-1)dL_{s}^{0}(M^{\alpha})}.$$
Since $\langle D,M\rangle=0$, it follows from Corollary \ref{car} that $dv_{t}$ is carried by $H$. Hence,
$$\int_{0}^{t}{\mathcal{Z}^{\alpha}_{s}dv_{s}}=0$$
since $H\subset\{t\geq0: M_{t}=0\}$ and the process $\mathcal{Z}^{\alpha}$ is defined on the complementary of the set $\{t\geq0: M_{t}=0\}$. Then,
$$M^{\alpha}_{t}=\int_{0}^{t}{\mathcal{Z}^{\alpha}_{s}dm_{s}}+\int_{0}^{t}{(2\alpha(s)-1)dL_{s}^{0}(M^{\alpha})}.$$
It is obvious to see that $W=\left(\int_{0}^{t}{\mathcal{Z}^{\alpha}_{s}dm_{s}:t\geq0}\right)$ is a $\P-$ local martingale since $m$ is too. Moreover,
$$\langle Y,Y\rangle_{t}=\int_{0}^{t}{(\mathcal{Z}^{\alpha}_{s})^{2}d\langle m,m\rangle_{s}}=\langle m,m\rangle_{t}$$
since $\mathcal{Z}^{\alpha}_{s}\in\{-1,1\}$. But, $\langle m,m\rangle_{t}=\langle M,M\rangle_{t}=t$. Consequently, $Y$ is a $\P-$ Brownian motion, ensuring that $M^{\alpha}$ satisfies \eqref{sbm}.
$$ $$
(2) We have from Proposition 2.4 of \cite{siam} that
$$|M|^{\alpha}_{t}=\int_{0}^{t}{\mathcal{Z}^{\alpha}_{s}d|M_{s}|}+\int_{0}^{t}{(2\alpha(s)-1)dL_{s}^{0}(|M|^{\alpha})}.$$
That is,
$$|M|^{\alpha}_{t}=\int_{0}^{t}{\mathcal{Z}^{\alpha}_{s}{\rm sgn(M_{s})}dM_{s}}+\int_{0}^{t}{\mathcal{Z}^{\alpha}_{s}dL^{0}_{s}(M)}+\int_{0}^{t}{(2\alpha(s)-1)dL_{s}^{0}(|M|^{\alpha})}.$$
But,
$$\int_{0}^{t}{\mathcal{Z}^{\alpha}_{s}dL^{0}_{s}(M)}=0$$
since $dL_{t}^{0}(M)$ is carried by $\{t\geq0:M_{t}=0\}$ and $\mathcal{Z}^{\alpha}_{s}$ is defined on the complementary of $\{t\geq0:M_{t}=0\}$. So, that entails that
$$|M|^{\alpha}_{t}=\int_{0}^{t}{\mathcal{Z}^{\alpha}_{s}dW_{s}}+\int_{0}^{t}{(2\alpha(s)-1)dL_{s}^{0}(|M|^{\alpha})}$$
where, $W_{t}=\int_{0}^{t}{{\rm sgn}(M_{s})dM_{s}}$. But, we can see from Corollary \ref{car4} that $W$ is a $\P-$ Brownian motion. That implies that $\left(\int_{0}^{t}{\mathcal{Z}^{\alpha}_{s}dW_{s}}:t\geq0\right)$ is a $\P-$ Brownian motion, ensuring that $|M|^{\alpha}$ satisfies \eqref{sbm}.
\end{proof}

\begin{rem}
Since every $\P$- Brownian motion $B$ null on $H$ such that $\langle D,B\rangle=0$ is also a $(\Qv,\P)$- Brownian motion. Hence, the solution 2. of the above theorem coincides with the solution of \cite{siam}.
\end{rem}

In what follows, we construct solutions of the time inhomogeneous skew Brownian motion equation with stochastic processes of the class  $\Sigma(H)$.

\begin{theorem}
Let $X=M+V$ be a continuous stochastic process of class $\Sigma(H)$ null on $H$ such that $\langle X,X\rangle_{t}=t$ and  $\langle X,D\rangle_{t}=0$. Then, the following hold:
\begin{enumerate}
	\item $X^{\alpha}=\mathcal{Z}^{\alpha}X$ is a weak solution of \eqref{isbm}.
	\item $|X|^{\alpha}=\mathcal{Z}^{\alpha}|X|$ is a weak solution of \eqref{isbm}.
\end{enumerate}
\end{theorem}
\begin{proof}
$(1)$ We have from Proposition 2.2 of \cite{siam} that 
$$X^{\alpha}_{t}=\int_{0}^{t}{\mathcal{Z}^{\alpha}_{s}dX_{s}}+\int_{0}^{t}{(2\alpha(s)-1)dL_{s}^{0}(X^{\alpha})}.$$
That is,
$$X^{\alpha}_{t}=\int_{0}^{t}{\mathcal{Z}^{\alpha}_{s}dM_{s}}+\int_{0}^{t}{\mathcal{Z}^{\alpha}_{s}dV_{s}}+\int_{0}^{t}{(2\alpha(s)-1)dL_{s}^{0}(X^{\alpha})}.$$
We know that $dV_{t}$ is carried by $\{t\geq0: X_{t}D_{t}=0\}$. But, $\{t\geq0: X_{t}D_{t}=0\}=\{t\geq0: X_{t}=0\}$ since $X$ vanishes on $H$. Hence, it follows that
$$X^{\alpha}_{t}=\int_{0}^{t}{\mathcal{Z}^{\alpha}_{s}dM_{s}}+\int_{0}^{t}{(2\alpha(s)-1)dL_{s}^{0}(X^{\alpha})}.$$
Moreover, the $(\Qv,\P)-$ local martingale $M=m+v$ satisfies $\langle M,D\rangle_{t}=\langle X,D\rangle_{t}=0$. Then, it follows from Corollary \ref{car} that $dv_{t}$ is carried by $H\subset \{t\geq0: X_{t}=0\}$. Therefore,
$$ \int_{0}^{t}{\mathcal{Z}^{\alpha}_{s}dv_{s}}=0.$$
Thus,
$$X^{\alpha}_{t}=\int_{0}^{t}{\mathcal{Z}^{\alpha}_{s}dm_{s}}+\int_{0}^{t}{(2\alpha(s)-1)dL_{s}^{0}(X^{\alpha})}.$$
It is obvious to see that $Y=\left(\int_{0}^{t}{\mathcal{Z}^{\alpha}_{s}dm_{s}}: t\geq0\right)$ is a $\P-$ local martingale. Furthermore, 
$$\langle Y,Y\rangle_{t}=\langle m,m\rangle_{t}=\langle M,M\rangle_{t}=\langle X,X\rangle_{t}=t.$$
Consequently, $Y$ is a $\P-$ Brownian motion ensuring that $X^{\alpha}$ satisfies  \eqref{sbm}.
$$ $$
$(2)$ One has, 
$$|X|^{\alpha}_{t}=\int_{0}^{t}{\mathcal{Z}^{\alpha}_{s}d|X_{s}|}+\int_{0}^{t}{(2\alpha(s)-1)dL_{s}^{0}(|X|^{\alpha})}$$
$$\hspace{5cm}=\int_{0}^{t}{\mathcal{Z}^{\alpha}_{s}{\rm sgn}(X_{s})dX_{s}}+\int_{0}^{t}{\mathcal{Z}^{\alpha}_{s}dL^{0}_{s}(X)}+\int_{0}^{t}{(2\alpha(s)-1)dL_{s}^{0}(|X|^{\alpha})}.$$
But,
$$\int_{0}^{t}{\mathcal{Z}^{\alpha}_{s}dL^{0}_{s}(X)}=0$$
since $dL_{t}^{0}(X)$ is carried by $\{t\geq0: X_{t}=0\}$ and $\mathcal{Z}^{\alpha}$ is defined on the complementary of $\{t\geq0: X_{t}=0\}$. Hence,
$$|X|^{\alpha}_{t}=\int_{0}^{t}{\mathcal{Z}^{\alpha}_{s}{\rm sgn}(X_{s})dX_{s}}+\int_{0}^{t}{(2\alpha(s)-1)dL_{s}^{0}(|X|^{\alpha})}$$
That implies,
$$|X|^{\alpha}_{t}=\int_{0}^{t}{\mathcal{Z}^{\alpha}_{s}{\rm sgn}(X_{s})dM_{s}}+\int_{0}^{t}{\mathcal{Z}^{\alpha}_{s}{\rm sgn}(X_{s})dV_{s}}+\int_{0}^{t}{(2\alpha(s)-1)dL_{s}^{0}(|X|^{\alpha})}.$$
We have 
$$\int_{0}^{t}{\mathcal{Z}^{\alpha}_{s}{\rm sgn}(X_{s})dV_{s}}=0$$
because $dV_{t}$ is carried by $\{t\geq0: X_{t}=0\}$. Thus,
$$|X|^{\alpha}_{t}=\int_{0}^{t}{\mathcal{Z}^{\alpha}_{s}{\rm sgn}(X_{s})dM_{s}}+\int_{0}^{t}{(2\alpha(s)-1)dL_{s}^{0}(|X|^{\alpha})}$$
$$\hspace{4.5cm}=\int_{0}^{t}{\mathcal{Z}^{\alpha}_{s}{\rm sgn}(X_{s})dm_{s}}+\int_{0}^{t}{\mathcal{Z}^{\alpha}_{s}{\rm sgn}(X_{s})dv_{s}}+\int_{0}^{t}{(2\alpha(s)-1)dL_{s}^{0}(|X|^{\alpha})}.$$
One has, $\langle M,D\rangle_{t}=\langle X,D\rangle_{t}=0$. Hence, we obtain from Corollary \ref{car} that $dv_{t}$ is carried by $H\subset\{t\geq0: X_{t}=0\}$. Then,
$$\int_{0}^{t}{\mathcal{Z}^{\alpha}_{s}{\rm sgn}(X_{s})dv_{s}}=0.$$
That is, 
$$|X|^{\alpha}_{t}=\int_{0}^{t}{\mathcal{Z}^{\alpha}_{s}{\rm sgn}(X_{s})dm_{s}}+\int_{0}^{t}{(2\alpha(s)-1)dL_{s}^{0}(|X|^{\alpha})}.$$
$W=\left(\int_{0}^{t}{\mathcal{Z}^{\alpha}_{s}{\rm sgn}(X_{s})dm_{s}}:t\geq0\right)$ is a $\P-$ local martingale and 
$$\langle W,W\rangle_{t}=\langle m,m\rangle_{t}=\langle M,M\rangle_{t}=\langle X,X\rangle_{t}=t.$$
Consequently, $W$ is a $\P-$ Brownian motion ensuring that $|X|^{\alpha}$ satisfies  \eqref{sbm}.
\end{proof}

{\color{myaqua}

\end{document}